\documentclass[11pt]{amsart}

\usepackage{graphicx}
\usepackage{epsfig}
\usepackage{amsmath}
\usepackage{amssymb}
\usepackage{bigints}
\usepackage{tikz}
\usepackage{graphicx}
\usepackage{epsfig}
\usepackage[noadjust]{cite}
\usepackage{xcolor}

\usepackage[colorlinks,citecolor=red,urlcolor=blue,bookmarks=false,hypertexnames=true]{hyperref}

\theoremstyle{plain}
\newtheorem{theorem}                {Theorem}      [section]
\newtheorem*{theorem*}                {Theorem \ref{thm:appl}}
\newtheorem{proposition}  [theorem]  {Proposition}

\newtheorem{lemma}        [theorem]  {Lemma}

\theoremstyle{definition}

\newtheorem{remark}       [theorem]  {Remark}

\DeclareMathOperator{\trace}{trace} 

\DeclareMathOperator{\Div}{div}

\DeclareMathOperator{\grad}{grad}

\DeclareMathOperator{\Int}{Int}

\numberwithin{equation}{section}

\begin{document}

\title[Uniqueness of complete biconservative surfaces in $N^3(\epsilon)$]
{On the uniqueness of complete biconservative surfaces in $3$-dimensional space forms}

\author{Simona~Nistor, Cezar~Oniciuc}

\address{Faculty of Mathematics\\ Al. I. Cuza University of Iasi\\
Blvd. Carol I, 11 \\ 700506 Iasi, Romania} \email{nistor.simona@ymail.com}

\address{Faculty of Mathematics\\ Al. I. Cuza University of Iasi\\
Blvd. Carol I, 11 \\ 700506 Iasi, Romania} \email{oniciucc@uaic.ro}

\thanks{This work was supported by a grant of the Romanian Ministry of Research and Innovation, CCCDI-UEFISCDI, project number PN-III-P3-3.1-PM-RO-FR-2019-0234 / 1BM / 2019, whithin PNCDI III}



\begin{abstract}
Biconservative surfaces are surfaces with divergence-free stress-bienergy tensor. Simply connected, complete, non-$CMC$ biconservative surfaces in $3$-dimensional space forms were constructed working in extrinsic and intrinsic ways. Then, one raises the question of the uniqueness of such surfaces. In this paper we give a positive answer to this question.
\end{abstract}

\maketitle
\section{Introduction}

Biharmonic submanifolds can be viewed as a generalization of minimal submanifolds, and in the last two decades they have been studied intensively. The characterizing equation of biharmonic submanifolds, i.e., the biharmonic equation, can be decomposed in its tangent and normal parts.

The biharmonic submanifolds are rather rigid, while the biconservative ones are obtained by relaxing the biharmonicity condition. More precisely, the biconservative submanifolds are defined by $\Div S_2=0$, where $S_2$ is the stress-bienergy tensor, and are characterized by the tangent part of the biharmonic equation. Thus, we have more examples of biconservative submanifolds, also when biharmonic submanifolds do not exist (see \cite{FOP15,FT16,J87,LMO08,MOR16,S15,T15,YT18}).

Despite their characterization being simpler than in the biharmonic case, the biconservative submanifolds have interesting geometric properties. A most remarkable one is the fact that the Hopf differential associated to a biconservative surface in an arbitrary Riemannian manifold is holomorphic if and only if the surface is $CMC$, i.e. it has constant mean curvature.

Biconservative hypersurfaces in space forms $N^{m+1}(\epsilon)$ of constant sectional curvature $\epsilon$ are characterized by
$$
A(\grad f)=-\frac{f}{2}\grad f,
$$
where $A$ is the shape operator and $f=\trace A$ is the mean curvature function.

Obviously, any $CMC$ hypersurface in a space form is biconservative, and therefore, when working in space forms, one is interested in the study of non-$CMC$ biconservative hypersurfaces, i.e. $\grad f\neq 0$ at any point of an open subset.

The explicit parametric equations of biconservative surfaces in $N^3(\epsilon)$ with $\grad f\neq 0$ at any point have been found in \cite{CMOP14, F15, HV95} and then, using these equations, simply connected, complete surfaces with $\grad f\neq 0$ on a dense subset have been constructed. These constructions were achieved in both extrinsic and intrinsic approaches. The intrinsic way is to first construct the abstract domain of the immersion, while the extrinsic one is to work with their explicit parametric equations (see \cite{NPhD17, N16, NO19-H}). The following issue appears naturally:

\textbf{Open question.} \textit{The only simply connected, complete, non-$CMC$ biconservative surfaces in $N^3(\epsilon)$ are the ones given above, up to isometries of the domain and the codomain.}

The answer to the open question is positive when considering complete, regular, biconservative surfaces in $3$-dimensional Euclidian space $\mathbb{R}^3$ (see \cite{NO19}).

In this paper, we show that if $\epsilon\in\{-1,0,1\}$, then the answer is also positive, even if we drop the regularity condition, that is the immersions do not have to be homeomorphisms on their images.

Our paper is organized as follows. In Section \ref{sect:GenProp} we review some known results on the properties of biconservative surfaces in $N^3(\epsilon)$ that will be further used.

In Section \ref{sec-complete} we follow the same idea as in \cite{NO19-H}, where the ambient space is the $3$-dimensional hyperbolic space $\mathbb{H}^3$, to construct, in an intrinsic way, simply connected, complete, non-$CMC$ biconservative surfaces in the $3$-dimensional unit Euclidean sphere $\Phi_{1,C}:\left(\mathbb{R}^2,\tilde{g}_{1,C}\right) \to \mathbb{S}^3$, where $C>4/\sqrt{3}$ is a real parameter. More precisely, writing the metric of the abstract standard biconservative surfaces in a convenient way, we first construct the abstract domains of the biconservative immersions, i.e., $\left(\mathbb{R}^2,\tilde{g}_{1,C}\right)$, and the shape operator. Then we use the Fundamental Theorem of Surfaces in $\mathbb{S}^3$ in order to prove the existence and the uniqueness of $\Phi_{1,C}$ (see Theorem \ref{th:EUS3}). The difficult part is that, in order to obtain the abstract domains of the immersions we are looking for, one has to glue an infinite number of abstract standard biconservative surfaces. As the above method also works when $\epsilon=0$, we get a unitary procedure to construct simply connected, complete, non-$CMC$ biconservative surfaces in any $N^3(\epsilon)$.

In the last section we prove that the simply connected, complete, non-$CMC$ biconservative surfaces in $N^3(\epsilon)$ obtained before are the only ones with these properties. The idea of the proof is to show that if $\Phi:M^2\to N^3(\epsilon)$ is a simply connected, complete, non-$CMC$ biconservative surface then $M$ must be isometric to $\left(\mathbb{R}^2,\tilde{g}_{\epsilon,C}\right)$, for some constant $C$, and then to use a theorem of the same type as Theorem \ref{th:EUS3} in order to conclude with $\Phi=\Phi_{\epsilon,C}$.

\textbf{Conventions.} We assume that all manifolds are connected and oriented, and use the following sign conventions for the rough Laplacian acting on sections of $\varphi^{-1}(TN)$ and for the curvature tensor field of $N$, respectively:
$$
\Delta^{\varphi}=-\trace_{g} \left(\nabla^{\varphi}\nabla^{\varphi}-\nabla^{\varphi}_{\nabla}\right)
$$
and
$$
R^N(X,Y)Z=[\nabla^N_X,\nabla^N_Y]Z-\nabla^N_{[X,Y]}Z.
$$

\section{General properties of biconservative surfaces in $N^3(\epsilon)$ }\label{sect:GenProp}

In this section we recall some properties of biconservative surfaces in space forms with nowhere vanishing $\grad f$.

\begin{theorem}[\cite{CMOP14}]\label{thm:CMOP}
Let $\varphi:M^2\to N^3(\epsilon)$ be a biconservative surface with $\grad f\neq 0$ at any point of $M$. Then the Gaussian curvature $K$ satisfies
\begin{itemize}
\item [(i)]
    $$
    K=\det A+\epsilon=-\frac{3f^2}{4}+\epsilon;
    $$
\item [(ii)] $\epsilon-K>0$, $\grad K\neq 0$ at any point of $M$, and the level curves of $K$ are circles in $M$ with constant curvature
\begin{equation*}
\kappa=\frac{3|\grad K|}{8(\epsilon-K)};
\end{equation*}
\item [(iii)]
$$
(\epsilon-K)\Delta K-|\grad K|^2-\frac{8}{3}K(\epsilon-K)^2=0,
$$
where $\Delta$ is the Laplace-Beltrami operator on $M$.
\end{itemize}
\end{theorem}

In particular, it follows that, choosing $H/|H|$ as the unit normal vector field, we have $f>0$.

\begin{theorem}[\cite{CMOP14}]
Let $M^2$ be a biconservative surface in $N^3(\epsilon)$ with nowhere vanishing gradient of the mean curvature function $f$. Then
\begin{equation}\label{eq:bicons1}
f\Delta f+|\grad f|^2+\frac{4}{3}\epsilon f^2-f^4=0,
\end{equation}
where $\Delta$ is the Laplace-Beltrami operator on $M$.
\end{theorem}

We recall now a classical result concerning the existence of $CMC$ surfaces in $N^3(\epsilon)$.

\begin{theorem}[\cite{L70}]\label{thm:cmc-holom}
Let $\varphi:\left(M^2,g\right)\to N^3(\epsilon)$ be a $CMC$ surface. Then $|H|^2+\epsilon-K\geq 0$ at any point, and either $|H|^2+\epsilon-K=0$ everywhere, i.e., $M$ is umbilical, or $|H|^2+\epsilon-K=0$ only at isolated points. Moreover, on the set where $|H|^2+\epsilon-K>0$, we have
\begin{equation*}
\Delta \log\left(|H|^2+\epsilon-K\right)+4K=0,
\end{equation*}
or, equivalently,
\begin{equation}\label{h-holomorphic-equiv}
\left(|H|^2+\epsilon-K\right)\Delta K-|\grad K|^2-4K\left(|H|^2+\epsilon-K\right)^2=0.
\end{equation}
\end{theorem}

\begin{remark}[\cite{R95}]
If $\varphi:\left(M^2,g\right)\to N^3(\epsilon)$ is a minimal surface, i.e., $H=0$, the conclusions of Theorem \ref{thm:cmc-holom} are still valid.
\end{remark}

Next, we recall a uniqueness result, that was first stated, as a remark, in \cite{FNO16}.

\begin{theorem}[\cite{FNO16}]\label{thm:uniqueness}
Let $\left(M^2,g\right)$ be an abstract surface and $\epsilon\in\mathbb{R}$ an arbitrarily fixed constant. If $M$ admits two biconservative immersions in $N^3(\epsilon)$, both with nowhere vanishing gradient of the mean curvature function, then the two immersions differ by an isometry of $N^3(\epsilon)$.
\end{theorem}

In the same paper \cite{FNO16}, the authors proved the following characterization theorem for biconservative surfaces in $N^3(\epsilon)$ and found some properties of them.

\begin{theorem}[\cite{FNO16}]\label{thm:char}
Let $\left(M^2,g\right)$ be an abstract surface. Then $M$ can be locally isometrically embedded in a space form $N^3(\epsilon)$ as a biconservative surface with the gradient of the mean curvature function different from zero everywhere if and only if the Gaussian curvature $K$ satisfies $\epsilon-K(p)>0$, $(\grad K)(p)\neq 0$, for any $p\in M$, and its level curves are circles in $M$ with constant curvature
$$
\kappa=\frac{3|\grad K|}{8(\epsilon-K)}.
$$
\end{theorem}

\begin{theorem}[\cite{FNO16}]\label{thm:metric}
Let $\left(M^2,g\right)$ be an abstract surface with the Gaussian curvature $K$ and $\epsilon\in\mathbb{R}$ an arbitrarily fixed constant, and assume that $(\grad K)(p)\neq 0$ and $\epsilon-K(p)>0$ at any point $p\in M$. Let $X_1=\grad K/|\grad K|$ and $X_2\in C(TM)$ be two vector fields on $M$ such that $\{X_1(p),X_2(p)\}$ is a positively oriented orthonormal basis at any point $p\in M$. If the level curves of $K$ are circles in $M$ with constant curvature
$$
\kappa=\frac{3X_1K}{8(\epsilon-K)}=\frac{3|\grad K|}{8(\epsilon-K)},
$$
then, for any point $p_0\in M$, there exists a positively oriented parametrization $X=X(u,v)$ of $M$ in a neighborhood $U\subset M$ of $p_0$  such that
\begin{itemize}
\item[(i)] the curve $u\to X(u,0)$ is an integral curve of $X_1$ with $X(0,0)=p_0$ and $v\to X(u,v)$ is an integral curve of $X_2$, for any $u$ and $v$;

\item[(ii)] $K(u,v)=(K\circ X)(u,v)=(K\circ X)(u,0)=K(u)$, for any $(u,v)$;

\item[(iii)] for any pair $(u,v)$, we have
\begin{align*}
g_{11}(u,v)& = \frac{9}{64}\left(\frac{K^\prime(u)}{\epsilon-K(u)}\right)^2v^2+1,\\
g_{12}(u,v)& = -\frac{3K^\prime(u)}{8(\epsilon-K(u))}v,\quad g_{22}(u,v)=1;
\end{align*}
\item[(iv)] the Gaussian curvature $K=K(u)$ satisfies
$$
24(\epsilon-K)K^{\prime\prime}+33(K^\prime)^2+64K(\epsilon-K)^2=0.
$$
\end{itemize}
\end{theorem}

Now, using Theorems \ref{thm:uniqueness} and \ref{thm:metric}, we can prove that any biconservative immersion from the above abstract domain in $N^3(\epsilon)$ has the property $\grad f\neq 0$ at any point. More precisely, we have

\begin{theorem}[\cite{NPhD17}]\label{thm:CMC-bicons}
Let $\left(M^2,g\right)$ be an abstract surface and $\epsilon\in\mathbb{R}$ an arbitrarily fixed constant. Assume that $\epsilon-K>0$ and $\grad K\neq0$ at any point of $M$, and the level curves of $K$ are circles in $M$ with constant curvature
$$
\kappa=\frac{3|\grad K|}{8(\epsilon-K)}.
$$
If there exists a biconservative immersion $\varphi:\left(M^2,g\right)\to N^3(\epsilon)$, then $\grad f\neq 0$ and $f>0$ at any point of $M$. Moreover, the immersion $\varphi$ is unique.
\end{theorem}

\begin{proof}
Assume that there exists a biconservative immersion $\varphi:\left(M^2,g\right)\to N^3(\epsilon)$. First, we will prove that $\grad f\neq 0$ at any point of an open dense subset of $M$. Indeed, if the set
$$
\Omega=\left\{p\in M \ :\ (\grad f)(p)\neq 0\right\}
$$
were not dense, then $\grad f$ would vanish on $M\setminus \overline{\Omega}$, which is an open, non-empty set. Let us denote by $V$ a connected component of $M\setminus \overline{\Omega}$. We note that $V$ is also open in $M$. Using Theorem \ref{thm:metric}, we have
\begin{equation}\label{eq:ec-K''}
24(\epsilon-K)K^{\prime\prime}+33\left(K^\prime\right)^2+64K(\epsilon-K)^2=0.
\end{equation}
On the other hand, as $\varphi$ is $CMC$ or minimal on $V$, using the same local coordinates $(u,v)$ as above, and the fact that $|H|^2+\epsilon-K>0$, equation \eqref{h-holomorphic-equiv} can be rewritten as
\begin{equation}\label{eq:ec1-K''}
8(\epsilon-K)K^{\prime\prime}+\left(\frac{8(\epsilon-K)}{|H|^2+\epsilon-K}+3\right)\left(K^\prime\right)^2+32K(\epsilon-K)\left(|H|^2+\epsilon-K\right)=0.
\end{equation}
From equations \eqref{eq:ec-K''} and \eqref{eq:ec1-K''}, one obtains
\begin{equation}\label{eq:double-derivative-K}
3|H|^2\left(K^\prime\right)^2-4K(\epsilon-K)\left(\left(|H|^2+\epsilon-K\right)^2+2|H|^2\left(|H|^2+\epsilon-K\right)\right)=0.
\end{equation}
We note that
$$
\left(K^\prime\right)^2=\frac{64}{3}K^3-\frac{640}{9}\epsilon K^2+\alpha(\epsilon-K)^{11/4}+\frac{704}{9}\epsilon^2K-\frac{256}{9}\epsilon^3,
$$
where $\alpha\in\mathbb{R}$ is a constant, is a first integral of \eqref{eq:ec-K''}. Now, if we replace $\left(K^\prime\right)^2$ in \eqref{eq:double-derivative-K}, one gets the following. If $\varphi$ is minimal, then $K$ has to satisfy a fourth order polynomial equation with constant coefficients, with the leading term $4K^4$, or, if $\varphi$ is $CMC$, we obtain that $K$ has to satisfy a $16$-th order polynomial equation with constant coefficients, with the leading term $256K^{16}$. In both situations, we come to the conclusion that $K$ has to be a constant, and this is a contradiction.

Thus, $\grad f\neq 0$ on $\Omega$, which is an open dense subset of $M$. From the Gauss equation, $K=\epsilon+\det A$, we obtain on $\Omega$ that
$$
f^2=\frac{4}{3}(\epsilon-K).
$$
As $\Omega$ is dense in $M$, it follows that, in fact, the above relation holds on whole $M$. Therefore, since $\epsilon-K>0$ and $\grad K\neq 0$ on $M$, one obtains $f>0$ and $\grad f\neq 0$ at any point of $M$.

Finally, the uniqueness of $\varphi$ follows from Theorem \ref{thm:uniqueness}.
\end{proof}

Now, using Theorems \ref{thm:char} and \ref{thm:CMC-bicons} we can state the following result.

\begin{theorem}[\cite{NPhD17}]\label{thm:reformulate}
Let $\left(M^2,g\right)$ be an abstract surface and $\epsilon\in\mathbb{R}$ an arbitrarily given constant. Assume that $\epsilon-K>0$ and $\grad K\neq0$ at any point of $M$, and the level curves of $K$ are circles in $M$ with constant curvature
$$
\kappa=\frac{3|\grad K|}{8(\epsilon-K)}.
$$
Then, locally, there exists a unique biconservative embedding $\varphi:\left(M^2,g\right)\to N^3(\epsilon)$. Moreover, the mean curvature function is positive and its gradient is different from zero at any point of $M$.
\end{theorem}

Below, we give some equivalent conditions with the hypothesis from the above theorem which will be very useful in the construction from the next section.

\begin{theorem}[\cite{FNO16,N16,NO17}]\label{thm:carac}
Let $\left(M^2,g\right)$ be an abstract surface and $\epsilon\in\mathbb{R}$ an arbitrarily fixed constant. Assume that $\epsilon-K(p)>0$ and $(\grad K)(p)\neq 0$ at any point $p\in M$. Let $X_1=\grad K /|\grad K|$ and $X_2\in C(TM)$ be two vector fields on $M$ such that $\left\{X_1(p),X_2(p)\right\}$ is a positively oriented orthonormal basis at any point $p\in M$. Then, the following conditions are equivalent:
\begin{itemize}
  \item [(i)] the level curves of $K$ are circles in $M$ with constant curvature
  $$
  \kappa=\frac{3|\grad K|}{8(\epsilon-K)}=\frac{3X_1K}{8(\epsilon-K)};
  $$
  \item [(ii)]
  $$
  X_2\left(X_1K\right)=0 \quad \text{and} \quad \nabla_{X_2}X_2=\frac{-3X_1K}{8(\epsilon-K)}X_1;
  $$
  \item [(iii)]
  $$
  \nabla_{X_1}X_1=\nabla_{X_1}X_2=0,\quad \nabla_{X_2}X_2=-\frac{3X_1K}{8(\epsilon-K)}X_1,\quad \nabla_{X_2}X_1=\frac{3X_1K}{8(\epsilon-K)}X_2.
  $$
  \item [(iv)] the metric $g$ can be locally written as $g=e^{2\sigma}\left(du^2+dv^2\right)$, where $(u,v)$ are positively oriented local coordinates, and $\sigma=\sigma(u)$ satisfies the equation
      $$
      \sigma^{\prime\prime}=e^{-2\sigma/3}-\epsilon e^{2\sigma}
      $$
      and the condition $\sigma^\prime>0$; moreover, the solutions of the above equation, $u=u(\sigma)$, are
      $$
      u=\int_{\sigma_0}^{\sigma}\frac{d\tau}{\sqrt{-3e^{-2\tau/3}-\epsilon e^{2\tau}+a}}+u_0,
      $$
      where $\sigma$ is in some open interval $I$, $\sigma_0\in I$ and $a,u_0\in \mathbb{R}$ are constants.
\end{itemize}
\end{theorem}

\section{Complete biconservative surfaces in $\mathbb{S}^3$}\label{sec-complete}

From Theorems \ref{thm:char} and \ref{thm:carac}, we have the following local intrinsic characterization of biconservative surfaces in $N^3(\epsilon)$: if we consider an abstract surface $\left(M^2,g\right)$ with $\epsilon-K(p)>0$ and $(\grad K)(p)\neq 0$, for any $p\in M$, then it locally admits a (unique) biconservative immersion in $N^3(\epsilon)$ with nowhere vanishing gradient of the mean curvature function if and only if the metric $g$ can be locally written as $g(u,v)=e^{2\sigma(u)}\left(du^2+dv^2\right)$, where $\sigma'(u)> 0$, for any $u$, and $u=u(\sigma)$ is given by
$$
u(\sigma)=\int_{\sigma_0}^{\sigma}\frac{d\tau}{\sqrt{-3e^{-2\tau/3}-\epsilon e^{2\tau}+a}}+u_0,
$$
$a$ and $u_0$ being real constants.

With the new coordinates $(\sigma,v)$ the metric $g$ can be written as
$$
g_{\epsilon,a}(\sigma,v)=e^{2\sigma}\left(\frac{1}{-3e^{-2\sigma/3}-\epsilon e^{2\sigma}+a}d\sigma^2+dv^2\right),
$$
and we have, for each $\epsilon$, a one parameter family of such metrics. In order to find a more convenient expression of the metric, we will change again the coordinates; considering $(\sigma,v)=\left(\log\left(3^{3/4}/\xi\right),\theta/3^{3/4}\right)$, $\xi>0$, $\theta\in\mathbb{R}$, and denoting $C=a\sqrt{3}/3\in\mathbb{R}$, one obtains
$$
g_{\epsilon,C}(\xi,\theta)=\frac{1}{\xi^2}\left(\frac{3}{-\xi^{8/3}+C\xi^2-3\epsilon}d\xi^2+d\theta^2\right),
$$
where $\theta\in\mathbb{R}$ and $\xi$ is positive and belongs to an open interval such that $T(\xi)=-\xi^{8/3}+C\xi^2-3\epsilon$ is positive. By a standard analysis, we determine the largest range of $\xi$ such that $T(\xi)>0$, and we come to the following conclusion.

\begin{theorem}
Let $\left(M^2,g(u,v)=e^{2\sigma(u)}\left(du^2+dv^2\right)\right)$ be an abstract surface, where $u=u(\sigma)$ is given by
$$
u(\sigma)=\int_{\sigma_0}^{\sigma}\frac{d\tau}{\sqrt{-3e^{-2\tau/3}-\epsilon e^{2\tau}+a}}+u_0, \qquad \sigma\in I,
$$
where $a$ and $u_0$ are real constants and $I$ is an open interval. Then $\left(M^2,g\right)$ is isometric to
$$
\left(D_{\epsilon,C},g_{\epsilon,C}\right)=\left(\left(\xi_{01},\xi_{02}\right)\times\mathbb{R}, g_{\epsilon,C}(\xi,\theta)=\frac{1}{\xi^2}\left(\frac{3}{-\xi^{8/3}+C\xi^2-3\epsilon}d\xi^2+d\theta^2\right)
\right),
$$
where $C$, $\xi_{01}$ and $\xi_{02}$ are as follows:
\begin{itemize}
  \item if $\epsilon=-1$, then
\begin{itemize}
    \item [(i)] if $C=0$, it follows that $\xi_{01}=0$, and $\xi_{02}=3^{3/8}$ is the vanishing point of $-\xi^{8/3}+3$;
    \item [(ii)] if $C<0$, it follows that $\xi_{01}=0$, and $\xi_{02}>0$ is the vanishing point of $-\xi^{8/3}+C\xi^2+3$;
    \item [(iii)] if $C>0$, it follows that $\xi_{01}=0$, and $\xi_{02}$ is the positive vanishing point of $-\xi^{8/3}+C\xi^2+3$ satisfying $\xi_{02}>\left(3C/4\right)^{3/2}$;
\end{itemize}
  \item if $\epsilon=0$, then $C>0$, $\xi_{01}=0$, and $\xi_{02}=C^{3/2}>\left(3C/4\right)^{3/2}$ is the positive vanishing point of $-\xi^{8/3}+C\xi^2$;
  \item if $\epsilon=1$, then $C>4/\sqrt{3}$ while $\xi_{01}\in\left(0,\left(3C/4\right)^{3/2}\right)$ and $\xi_{02}\in\left(\left(3C/4\right)^{3/2},\infty\right)$ are the vanishing points of $-\xi^{8/3}+C\xi^2-3$.
\end{itemize}
\end{theorem}

\begin{remark}
We call the surface $\left(D_{\epsilon,C},g_{\epsilon,C}\right)$ an abstract standard biconservative surface, and, for each $\epsilon$, we have a one-parameter family of abstract standard biconservative surfaces indexed by $C$.
\end{remark}

\begin{remark}
We note that, when $\epsilon=1$, we have
$$
\lim_{\xi\searrow\xi_{01}}\left|\frac{\partial}{\partial \xi}\right|^2= \lim_{\xi\nearrow\xi_{02}}\left|\frac{\partial}{\partial \xi}\right|^2=\infty,
$$
and therefore, the metric $g_{1,C}$ blows up at the boundary of $D_{1,C}$.
\end{remark}

The surface $\left(D_{\epsilon,C},g_{\epsilon,C}\right)$ is not complete since the geodesic $\theta=\theta_0$ cannot be defined on the whole $\mathbb{R}$, but only on a finite open interval. By standard computations it can be proved that its Gaussian curvature is given by
\begin{equation*}
K_{\epsilon,C}(\xi,\theta)=K_{\epsilon,C}(\xi)=-\frac{\xi^{8/3}}{9}+\epsilon.
\end{equation*}
Therefore,
\begin{equation}\label{eq:K-DC}
K_{\epsilon,C}'(\xi)=-\frac{8}{27}\xi^{5/3}<0
\end{equation}
and
\begin{equation*}
\grad K_{\epsilon,C}=\frac{\xi^2\left(-\xi^{8/3}+C\xi^2-3\epsilon\right)}{3} K'_{\epsilon,C}(\xi)\frac{\partial}{\partial \xi}
\end{equation*}
is nowhere vanishing on $D_{\epsilon,C}$. Obviously,
$$
\lim_{\xi\searrow\xi_{01}}(\grad K_{\epsilon,C})(\xi,\theta)=0\quad \text{and}\quad \lim_{\xi\nearrow\xi_{02}}(\grad K_{\epsilon,C})(\xi,\theta)=0, \qquad \theta\in\mathbb{R}.
$$
As the metric $g_{\epsilon,C}$ is not complete, as a way to obtain a complete one, denoted $\tilde{g}_{\epsilon,C}$, we will change one more time the coordinates and then we will glue, in a simple way, two or more (isometric) metrics $g_{\epsilon,C}$.

So, we will continue with the following change of coordinates
$$
(\xi,\theta)=\left(\xi_0(\rho),\theta\right),
$$
where $\xi_0=\xi_0(\rho)$ is the inverse function of $\rho_0$ which is given by
\begin{equation}\label{eq:rhozero}
\rho_0(\xi)=-\int_{\xi_{00}}^{\xi}\sqrt{\frac{3}{\tau^2\left(-\tau^{8/3}+C\tau^2-3\epsilon\right)}}\ d\tau,
\end{equation}
$\xi_{00}$ being an arbitrarily fixed constant in $\left(\xi_{01},\xi_{02}\right)$.

We are allowed to make the above change since $\rho_0'(\xi)<0$ for any $\xi$, i.e., $\rho_0$ is a strictly decreasing function. Moreover, we have the following lemma.

\begin{lemma}\label{lemma1}
Let the function $\rho_0:\left(\xi_{01},\xi_{02}\right)\to\mathbb{R}$ defined by \eqref{eq:rhozero}.
\begin{itemize}
  \item If $\epsilon=-1$, $\lim_{\xi\searrow \xi_{01}} \rho_0(\xi)= \lim_{\xi\searrow 0} \rho_0(\xi)=\infty$ and $\lim_{\xi\nearrow \xi_{02}} \rho_0(\xi)=\rho_{0,-1}$, where $\rho_{0,-1}$ is a negative real constant.
  \item If $\epsilon=0$, then $\lim_{\xi\searrow \xi_{01}} \rho_0(\xi)= \lim_{\xi\searrow 0} \rho_0(\xi)=\infty$ and $\lim_{\xi\nearrow \xi_{02}} \rho_0(\xi)=\rho_{0,-1}$, where $\rho_{0,-1}$ is a negative real constant (we preserve the same notation of the limit of $\rho_0$ when $\xi$ approaches $\xi_{02}$, as in the case $\epsilon=-1$).
  \item If $\epsilon=1$, $\lim_{\xi\searrow \xi_{01}} \rho_0(\xi)=\rho_{0,1}$ and $\lim_{\xi\nearrow \xi_{02}} \rho_0(\xi)=\rho_{0,-1}$, where $\rho_{0,1}$ is a positive real constant and $\rho_{0,-1}$ is a negative real constant.
\end{itemize}
\end{lemma}

\begin{proof}
We note that the proof of this result, when $\epsilon=-1$, was given in \cite{NO19-H}. It is easy to see that, by similar arguments, the case $\epsilon=0$ also holds.

From now on, we will consider the case $\epsilon=1$. In order to compute the limits of $\rho_0$ when $\xi$ approaches $\xi_{01}$ and $\xi_{02}$, respectively, first we note that $\rho_0(\xi)$ is positive for any $\xi\in \left(\xi_{01},\xi_{00}\right)$, and
$$
\rho_0(\xi)<-\frac{\sqrt{3}}{\xi_{01}}\int_{\xi_{00}}^{\xi}\frac{1}{\sqrt{-\tau^{8/3}+C\tau^2-3}}\ d\tau, \qquad \xi\in \left(\xi_{01},\xi_{00}\right).
$$
Then, since
$$
\lim_{\xi\searrow\xi_{01}} \sqrt{\xi-\xi_{01}}\cdot\frac{1}{\sqrt{-\xi^{8/3}+C\xi^2-3}}=\sqrt{\frac{3}{-8\xi_{01}^{5/3}+6C\xi_{01}}}\in[0,\infty),
$$
we get
$$
\lim_{\xi\searrow\xi_{01}}\int_{\xi_{00}}^{\xi}\frac{1}{\sqrt{-\tau^{8/3}+C\tau^2-3}}\ d\tau>-\infty.
$$
Therefore,
$$
\lim_{\xi\searrow \xi_{01}} \rho_0(\xi)=\rho_{0,1}\in\mathbb{R}^\ast_{+}.
$$
Similarly, we can see that $\rho_0(\xi)$ is negative for any $\xi\in \left(\xi_{00},\xi_{02}\right)$ and
$$
\rho_0(\xi)> -\frac{\sqrt{3}}{\xi_{00}}\int_{\xi_{00}}^{\xi}\frac{1}{\sqrt{-\tau^{8/3}+C\tau^2-3}}\ d\tau, \qquad \xi\in \left(\xi_{00},\xi_{02}\right).
$$
Thus, the limit of $\rho_0$ when $\xi$ approaches $\xi_{02}$ if finite, since
$$
\lim_{\xi\nearrow\xi_{02}} \sqrt{\xi_{02}-\xi}\cdot\frac{1}{\sqrt{-\xi^{8/3}+C\xi^2-3}}=\sqrt{\frac{3}{8\xi_{02}^{5/3}-6C\xi_{02}}}\in[0,\infty),
$$
and so
$$
\lim_{\xi\nearrow\xi_{02}}\int_{\xi_{00}}^{\xi}\frac{1}{\sqrt{-\tau^{8/3}+C\tau^2-3}}\ d\tau<\infty.
$$
In conclusion,
$$
\lim_{\xi\nearrow \xi_{02}} \rho_0(\xi)=\rho_{0,-1}\in\mathbb{R}^\ast_{-}.
$$
\end{proof}

With the above change of coordinates, the metric $g_{\epsilon,C}$ can be rewritten as
$$
g_{\epsilon,C}(\rho,\theta)=\frac{1}{\xi_{0}^2(\rho)}d\theta^2+d\rho^2,
$$
where
\begin{itemize}
  \item $(\rho,\theta)\in\left(\rho_{0,-1},\infty\right)\times\mathbb{R}$, with $\rho_{0,-1}<0$, if $\epsilon=-1$;
  \item $(\rho,\theta)\in\left(\rho_{0,-1},\infty\right)\times\mathbb{R}$, with $\rho_{0,-1}<0$, if $\epsilon=0$;
  \item $(\rho,\theta)\in\left(\rho_{0,-1},\rho_{0,1}\right)\times\mathbb{R}$, with $\rho_{0,-1}<0$ and $\rho_{0,1}>0$, if $\epsilon=1$;
\end{itemize}

\begin{remark}
We note that, when $\epsilon=1$, we have
$$
\lim_{\rho\searrow\rho_{0,-1}}\left|\frac{\partial}{\partial \theta}\right|^2=\frac{1}{\xi_{02}^2}\in\mathbb{R}^\ast_{+} \qquad \text{and} \qquad \lim_{\rho\nearrow\rho_{0,1}}\left|\frac{\partial}{\partial \theta}\right|^2=\frac{1}{\xi_{01}^2}\in\mathbb{R}^\ast_{+},
$$
and thus, the metric $g_{1,C}$ can be smoothly extended to the boundary $\rho=\rho_{0,-1}$ and $\rho=\rho_{0,1}$.
\end{remark}

In \cite{NO19-H} we constructed a family of complete metrics $\tilde{g}_{\epsilon,C}$, when $\epsilon=-1$. In that case, it was enough to glue two metrics $g_{-1,C}$ (the same real constant $C$), along the boundary, in order to obtain a complete surface $\left(\mathbb{R}^2,\tilde{g}_{-1,C}\right)$. Then, we prove that from $\left(\mathbb{R}^2,\tilde{g}_{-1,C}\right)$ there exists a unique biconservative immersion in $\mathbb{H}^3$. The same steps can be also taken when $c=0$. More precisely, if we glue two metrics $g_{0,C}$ (the same positive real constant $C$), one gets a complete surface $\left(\mathbb{R}^2,\tilde{g}_{0,C}\right)$; then, the existence and the uniqueness of a biconservative immersion from $\left(\mathbb{R}^2,\tilde{g}_{0,C}\right)$ in $\mathbb{R}^3$ can be proved using the Fundamental Theorem of Surfaces in $\mathbb{R}^3$.

\begin{remark}
In fact, for $\epsilon=0$ we could reobtain Theorem 4.1 from \cite{N16} (where the complete abstract surface was obtained by working with isothermal coordinates), and for $\epsilon=1$ we could reobtain Proposition 4.17  from \cite{N16} (where the idea was to notice that the abstract standard biconservative surface is isometric to a surface of revolution in $\mathbb{R}^3$).
\end{remark}

Further, we will focus on the case $\epsilon=1$. In order to construct a complete metric $\tilde{g}_{1,C}$, $C>4/\sqrt{3}$, we can follow the same ideas as in \cite{NO19-H}, but in this case the gluing process has to be performed for infinitely many times.

First we extend the surface $\left(\left(\rho_{0,-1},\rho_{0,1}\right)\times\mathbb{R}, g_{1,C}\right)$ by ``symmetry'' with respect to its boundary given by $\rho=\rho_{0,1}$ and $\rho=\rho_{0,-1}$, and then, continue this process for infinitely many times, in order to obtain a complete surface.

More precisely, we extend the surface $\left(\left(\rho_{0,-1},\rho_{0,1}\right)\times\mathbb{R},g_{1,C}\right)$ to the ``right hand side'', i.e.,  we impose the line $\rho=\rho_{0,1}$ to be an axis of symmetry. Therefore, we have $2\rho_{0,1}=\rho_0(\xi)+\rho_1(\xi)$, or, equivalently,  $\rho_1(\xi)=2\rho_{0,1}-\rho_0(\xi)$, where $\rho_1:\left(\xi_{01},\xi_{02}\right)\to\mathbb{R}$. It is easy to see that
$$
\lim_{\xi\searrow \xi_{01}} \rho_1(\xi)=\rho_{0,1}, \qquad \lim_{\xi\nearrow \xi_{02}}\rho_1(\xi) =2\rho_{0,1}-\rho_{0,-1},
$$
and, since $\rho_1'(\xi)=-\rho_0'(\xi)>0$, for any $\xi\in(\xi_{01},\xi_{02})$, it follows that $\rho_1$ is strictly increasing and the image $\rho_1\left(\xi_{01},\xi_{02}\right)$ is $\left(\rho_{0,1},\rho_{0,2}\right)$, where $\rho_{0,2}=2\rho_{0,1}-\rho_{0,-1}$.

Since $\rho_1$ is a diffeomorphism on its image, we can consider $\rho^{-1}_1:\left(\rho_{0,1},\rho_{0,2}\right) \to \left(\xi_{01},\xi_{02}\right)$, with $\rho^{-1}_1$ $:$ $\xi_1=\xi_1(\rho)$, $ \rho\in \left(\rho_{0,1},\rho_{0,2}\right)$.

Clearly,
$$
\lim_{\rho \searrow \rho_{0,1}}\xi_1(\rho)=\xi_{01}, \quad \lim_{\rho\nearrow {\rho_{0,2}}} \xi_1(\rho)=\xi_{02}.
$$
We also note that the new surface
$$
\left(\left(\rho_{0,1},\rho_{0,2}\right)\times\mathbb{R},\frac{1}{\xi_{1}^2(\rho)}d\theta^2+d\rho^2\right)
$$
is isometric to the initial surface
$$
\left(\left(\rho_{0,-1},\rho_{0,1}\right)\times\mathbb{R},\frac{1}{\xi_{0}^2(\rho)}d\theta^2+d\rho^2\right),
$$
and, therefore, it is also isometric to the abstract standard biconservative surface.

In order to extend our surface to the left hand side, we impose the line $\rho=\rho_{0,-1}$ to be an axis of symmetry. Therefore, we have $2\rho_{0,-1}=\rho_0(\xi)+\rho_{-1}(\xi)$, or, equivalently, $\rho_{-1}(\xi)=2\rho_{0,-1}-\rho_0(\xi)$, where $\rho_{-1}:\left(\xi_{01},\xi_{02}\right)\to\mathbb{R}$. It is easy to see that
$$
\lim_{\xi\searrow \xi_{01}}\rho_{-1}(\xi) =2\rho_{0,-1}-\rho_{0,1}, \qquad \lim_{\xi\nearrow \xi_{02}} \rho_{-1}(\xi)=\rho_{0,-1},
$$
and, since $\rho_{-1}'(\xi)=-\rho_0'(\xi)>0$, for any $\xi\in(\xi_{01},\xi_{02})$, it follows that $\rho_{-1}$ is strictly increasing and $\rho_{-1}\left(\xi_{01},\xi_{02}\right)=\left(\rho_{0,-2}, \rho_{0,-1}\right)$, where $\rho_{0,-2}=2\rho_{0,-1}-\rho_{0,1}$.

Since $\rho_{-1}$ is a diffeomorphism on its image, we can consider $\rho^{-1}_{-1}:\left(\rho_{0,-2}, \rho_{0,-1}\right) \to \left(\xi_{01},\xi_{02}\right)$, with $\rho^{-1}_{-1}$ $:$ $\xi_{-1}=\xi_{-1}(\rho)$, $ \rho\in \left(\rho_{0,-2}, \rho_{0,-1}\right)$.

Clearly,
$$
\lim_{\rho\searrow \rho_{0,-2}}\xi_{-1}(\rho)=\xi_{01}, \qquad \lim_{\rho\nearrow {\rho_{0,-1}}} \xi_{-1}(\rho)=\xi_{02}
$$
and we get a new surface isometric to the initial one.

We will continue this process, extending by symmetry with respect to the line $\rho=\rho_{0,r}$, for any $r\in \mathbb{Z}^\ast$. We define $\rho_{0,2}=2\rho_{0,1}-\rho_{0,-1}$; $\rho_{0,3}=2\rho_{0,2}-\rho_{0,1}=3\rho_{0,1}-2\rho_{0,-1}$, etc.; then $\rho_{0,-2}=2\rho_{0,-1}-\rho_{0,1}$; $\rho_{0,-3}=2\rho_{0,-2}-\rho_{0,-1}=3\rho_{0,-1}-2\rho_{0,1}$, etc.. Now, we can generalize these formulas and define $\rho=\rho_{0,r}$, $r\in \mathbb{Z}^\ast$, by
\begin{equation*}
\rho_{0,r}=\left\{
\begin{array}{ll}
  r\rho_{0,1}-(r-1)\rho_{0,-1}, & r\geq 1\\\\
  (r+1)\rho_{0,1}-r\rho_{0,-1}, & r\leq -1
\end{array}
\right..
\end{equation*}
Performing this process, we also obtain the functions $\rho_r$. For example, $\rho_1(\xi)=2\rho_{0,1}-\rho_0(\xi)$; $\rho_2(\xi)=2\rho_{0,2}-\rho_1(\xi)=2\rho_{0,1}-2\rho_{0,-1}+\rho_0(\xi)$, etc.; then $\rho_{-1}(\xi)=2\rho_{0,-1}-\rho_0(\xi)$; $\rho_{-2}(\xi)=2\rho_{0,-2}-\rho_{-1}(\xi)=2\rho_{0,-1}-2\rho_{0,1}+\rho_0(\xi)$, etc.. Generalizing, one obtains
\begin{equation*}
\rho_{r}(\xi)=\left\{
\begin{array}{ll}
  r\left(\rho_{0,1}-\rho_{0,-1}\right)+\rho_0(\xi),& r=2p\\\\
  (r+1)\rho_{0,1}-(r-1)\rho_{0,-1}-\rho_0(\xi), & r=2p+1
\end{array}
\right..
\end{equation*}
Since $\rho_s$ is a diffeomorphism on its image, we can consider $\xi_s=\xi_s(\rho)$ its inverse function, for any $s\in\mathbb{Z}$, and we note that
\begin{equation*}
\begin{array}{l}
\lim_{\rho\searrow \rho_{0,2p}}\xi_{2p}(\rho)=\xi_{02}, \qquad \lim_{\rho\searrow \rho_{0,2p-1}}\xi_{2p-1}(\rho)=\xi_{01}, \qquad p\geq 1,
\\\\
\lim_{\rho\nearrow \rho_{0,2p}}\xi_{2p}(\rho)=\xi_{01}, \qquad \lim_{\rho\nearrow \rho_{0,2p+1}}\xi_{2p+1}(\rho)=\xi_{02}, \qquad p\leq -1.
\end{array}
\end{equation*}
It is easy to see that
\begin{equation*}
\rho_{2s+1}'(\xi)>0, \qquad \rho_{2s}'(\xi)<0, \qquad  s\in\mathbb{Z},
\end{equation*}
and this is equivalent to
\begin{equation}\label{ec:DerivXiRho}
\xi_{2s+1}'(\rho)>0, \qquad \xi_{2s}'(\rho)<0, \qquad  s\in\mathbb{Z}.
\end{equation}
Now, we glue all functions $\xi_{s}$ and obtain the function $F:\mathbb{R}\to\left[\xi_{01},\xi_{02}\right]$ defined by
\begin{equation}\label{ec:gammma}
F(\rho)=\left\{
\begin{array}{ll}
  \xi_{02}, & \rho=\rho_{0,r}, \quad r=2p\in\mathbb{Z}_{+}, \quad p\geq 1\\\\
  \xi_{01}, & \rho=\rho_{0,r},  \quad r=2p-1\in\mathbb{Z}_{+}, \quad p\geq 1\\\\
  \xi_{r}(\rho), & \rho\in\left(\rho_{0,r},\rho_{0,r+1}\right), \quad r\in\mathbb{Z}_{+},\quad r\geq 1\\\\
  \xi_{0}(\rho), & \rho\in\left(\rho_{0,-1},\rho_{0,1}\right)  \\\\
  \xi_{r}(\rho), & \rho\in\left(\rho_{0,r-1},\rho_{0,r}\right),\quad r\in\mathbb{Z}_{-},\quad r\leq -1\\\\
  \xi_{01}, & \rho=\rho_{0,r},\quad  r=2p\in\mathbb{Z}_{-}, \quad p\leq -1\\\\
  \xi_{02}, & \rho=\rho_{0,r},\quad r=2p+1\in\mathbb{Z}_{-},\quad p\leq -1
\end{array}
\right..
\end{equation}
By some standard computations it is possible to verify that the non-vanishing function $F$ is at least of class $C^3$.

We also note that the function $F$ is periodic and the principal period is $2\left(\rho_{0,1}-\rho_{0,-1}\right)$. These properties of $F$ are illustrated in Figure \ref{Fig}, where we chose $C=3$.

\begin{figure}
  \centering
  \includegraphics[width=0.5\textwidth]{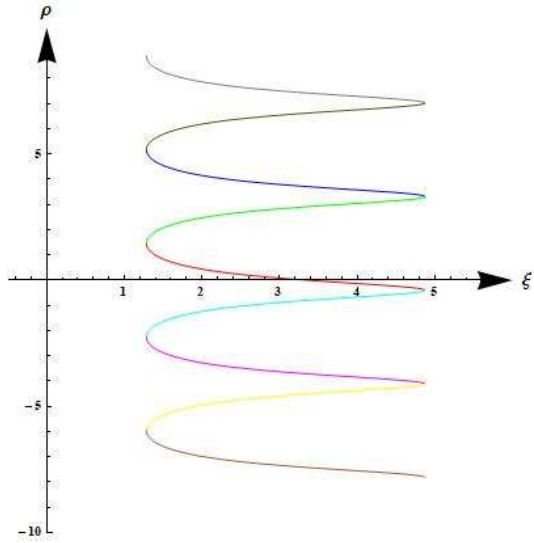}
  \caption{The graph of $F$}\label{Fig}
\end{figure}

In order to write in a simpler way the metric, we consider the function
$$
\Gamma:\mathbb{R}\to\left[\frac{1}{\xi_{02}},\frac{1}{\xi_{01}}\right],\qquad \Gamma(\rho)=\frac{1}{F(\rho)},
$$
which has the same properties as $F$.

In conclusion, we can state the next theorem.

\begin{theorem}\label{th-metricComplete}
The abstract surface
$$
\left(\mathbb{R}^2,\tilde{g}_{1,C}(\rho,\theta)=\Gamma^2(\rho)d\theta^2+d\rho^2\right)
$$
is complete.
\end{theorem}

\begin{proof}
In order to prove that the metric $\tilde{g}_{1,C}$ is complete, first we note that $\Gamma(\rho)\geq 1/\xi_{02}$, for any $\rho\in\mathbb{R}$, and then consider the metric
$$
\tilde{g}^0(\rho,\theta)=m_0\left(d\theta^2+d\rho^2\right), \qquad (\rho,\theta)\in\mathbb{R}^2,
$$
where $m_0$ is the minimum between $1/\xi_{02}^2$ and $1$. As the metric $\tilde{g}^0$ is complete and $\tilde{g}_{1,C}-\tilde{g}^0$ is non-negative at any point of the surface, it follows that $\tilde{g}_{1,C}$ is complete (see \cite{G73}).
\end{proof}

\begin{remark}\label{rk:geodesic}
Since $\left(\grad \tilde{K}_{1,C}\right)\left(\rho_{0,r},\theta\right)=0$, for any $\theta\in\mathbb{R}$, where $\tilde{K}_{1,C}$ is the Gaussian curvature of $\left(\mathbb{R}^2,\tilde{g}_{1,C}\right)$, it follows, from Theorem \ref{thm:carac} (ii), that the lines $\rho=\rho_{0,r}$ are geodesics in $\left(\mathbb{R}^2,\tilde{g}_{1,C}\right)$, for any $r\in\mathbb{Z}^\ast$.
\end{remark}

\begin{remark}
Since the Gaussian curvature of the complete surface $\left(\mathbb{R}^2,\tilde{g}_{1,C}\right)$ satisfies $\left(\grad \tilde{K}_{1,C}\right)\left(\rho_{0,r},\theta\right)=0$, for any $\theta\in\mathbb{R}$, one cannot simply apply Theorem \ref{thm:char} and therefore, the existence of a (non-$CMC$) biconservative immersion from $\left(\mathbb{R}^2,\tilde{g}_{1,C}\right)$ in $\mathbb{S}^3$ is not guaranteed. So, our aim is to overcome this difficulty and construct such an immersion.
\end{remark}

In the following construction of a globally defined orthonormal frame field on $\left(\mathbb{R}^2,\tilde{g}_{1,C}\right)$, for the sake of simplicity, we will omit to write the indices $1$ and $C$. So, let us denote $h_s(\rho)=1/\xi_s(\rho)$, for any $s\in\mathbb{Z}$, and consider
$$
^0g(\rho,\theta)=h_0^2(\rho)d\theta^2+d\rho^2,
$$
for any $(\rho,\theta)\in \left(\left(\rho_{0,-1},\rho_{0,1}\right)\times\mathbb{R}\right)$,
$$
^rg(\rho,\theta)=h_r^2(\rho)d\theta^2+d\rho^2,
$$
for any $(\rho,\theta)\in \left(\left(\rho_{0,r},\rho_{0,r+1}\right)\times\mathbb{R}\right)$, with $r\geq 1$, and
$$
^rg(\rho,\theta)=h_r^2(\rho)d\theta^2+d\rho^2,
$$
for any $(\rho,\theta)\in \left(\left(\rho_{0,r-1},\rho_{0,r}\right)\times\mathbb{R}\right)$, with $r\leq -1$.
It is easy to see that the Gaussian curvatures of the above surfaces are given by
$$
^sK(\rho)=-\frac{h_s''(\rho)}{h_s(\rho)}, \qquad s\in\mathbb{Z},
$$
and their derivatives are equal to
$$
^sK'(\rho)=\frac{-h_s'''(\rho)h_s(\rho)+h_s''(\rho)h_s'(\rho)}{h_s^2(\rho)}.
$$
We note that since $^sK(\rho)=K\left(\xi_s(\rho)\right)$, we have $^sK'(\rho)=K'\left(\xi_s(\rho)\right)\xi_s'(\rho)$ and then, using \eqref{eq:K-DC} and \eqref{ec:DerivXiRho}, it follows that $^sK'(\rho)<0$, if $s$ is odd, and $^sK'(\rho)>0$, if $s$ is even.

Now, let us consider the vector fields
$$
^sX_1=\frac{\grad \ ^sK}{\left|\grad \ ^sK\right|}, \qquad s\in\mathbb{Z},
$$
defined on $\left(\left(\rho_{0,-1},\rho_{0,1}\right)\times\mathbb{R}\right)$,  $\left(\left(\rho_{0,s},\rho_{0,s+1}\right)\times\mathbb{R}\right)$ or $\left(\left(\rho_{0,s-1},\rho_{0,s}\right)\times\mathbb{R}\right)$ if $s=0$, $s\geq 1$ or $s\leq -1$, respectively. It is easy to see that
$$
^sX_1=\frac{-h_s'''(\rho)h_s(\rho)+h_s''(\rho)h_s'(\rho)}{\left|-h_s'''(\rho)h_s(\rho)+h_s''(\rho)h_s'(\rho)\right|}\frac{\partial}{\partial \rho}
$$
and, since the sign of $^sK'(\rho)$ is given by the sign of the expression $-h_s'''(\rho)h_s(\rho)+h_s''(\rho)h_s'(\rho)$, one gets
$$
^sX_1=(-1)^s \frac{\partial}{\partial \rho}.
$$
Now, let us define the following vector field on $\mathbb{R}^2$
\begin{equation*}
X_1(\rho,\theta)=\left\{
\begin{array}{ll}
  ^rX_1(\rho,\theta), & (\rho,\theta)\in\left(\rho_{0,r},\rho_{0,r+1}\right)\times\mathbb{R}, \quad r=2p\in\mathbb{Z}_{+}, \quad p\geq 1\\\\
  -^rX_1(\rho,\theta), & (\rho,\theta)\in\left(\rho_{0,r},\rho_{0,r+1}\right)\times\mathbb{R},  \quad r=2p-1\in\mathbb{Z}_{+}, \quad p\geq 1\\\\
  ^0X_1(\rho,\theta), & (\rho,\theta)\in\left(\rho_{0,-1},\rho_{0,1}\right)\times\mathbb{R}\\\\
  ^rX_1(\rho,\theta), & (\rho,\theta)\in\left(\rho_{0,r-1},\rho_{0,r}\right)\times\mathbb{R},\quad r=2p\in\mathbb{Z}_{-}, \quad p\leq -1\\\\
 -^rX_1(\rho,\theta), & (\rho,\theta)\in\left(\rho_{0,r-1},\rho_{0,r}\right)\times\mathbb{R},\quad r=2p+1\in\mathbb{Z}_{-}, \quad p\leq -1\\\\
  \frac{\partial}{\partial \rho}, & (\rho,\theta)\in\{\rho_{0,r}\}\times\mathbb{R}, \quad r\in\mathbb{Z}^\ast
\end{array}
\right..
\end{equation*}
In fact, the vector field $X_1$ is given by $X_1=\frac{\partial}{\partial \rho}$ on $\mathbb{R}^2$.

Now, the vector field $X_1$ determines uniquely a global vector field $X_2$ by asking $\left\{X_1(\rho,\theta),X_2(\rho,\theta)\right\}$ to be a positive orthonormal frame field in $\left(\mathbb{R}^2,\tilde{g}\right)$, for any $(\rho,\theta)\in \mathbb{R}^2$. Obviously,
\begin{equation*}
X_2(\rho,\theta)=\left\{
\begin{array}{ll}
  ^rX_2(\rho,\theta), & (\rho,\theta)\in\left(\rho_{0,r},\rho_{0,r+1}\right)\times\mathbb{R}, \quad r=2p\in\mathbb{Z}_{+}, \quad p\geq 1\\\\
  -^rX_2(\rho,\theta), & (\rho,\theta)\in\left(\rho_{0,r},\rho_{0,r+1}\right)\times\mathbb{R},  \quad r=2p-1\in\mathbb{Z}_{+}, \quad p\geq 1\\\\
  ^0X_2(\rho,\theta), & (\rho,\theta)\in\left(\rho_{0,-1},\rho_{0,1}\right)\times\mathbb{R}\\\\
  ^rX_2(\rho,\theta), & (\rho,\theta)\in\left(\rho_{0,r-1},\rho_{0,r}\right)\times\mathbb{R},\quad r=2p\in\mathbb{Z}_{-}, \quad p\leq -1\\\\
 -^rX_2(\rho,\theta), & (\rho,\theta)\in\left(\rho_{0,r-1},\rho_{0,r}\right)\times\mathbb{R},\quad r=2p+1\in\mathbb{Z}_{-}, \quad p\leq -1\\\\
  \frac{1}{\Gamma\left(\rho_{0,r}\right)}\frac{\partial}{\partial \theta}, & (\rho,\theta)\in\{\rho_{0,r}\}\times\mathbb{R}, \quad r\in\mathbb{Z}^\ast
\end{array}
\right.,
\end{equation*}
that is $X_2$ is given by $X_2=\frac{1}{\Gamma(\rho)}\frac{\partial}{\partial \theta}$ on $\mathbb{R}^2$, where $1/\Gamma=F$ is given in \eqref{ec:gammma}.

In the following, we give some properties of $X_1$ and $X_2$.

\begin{proposition}\label{prop-1}
Let $\left(\mathbb{R}^2,\tilde{g}_{1,C}\right)$ the above complete surface. Then, the Gaussian curvature $\tilde{K}_{1,C}$ of $\left(\mathbb{R}^2,\tilde{g}_{1,C}\right)$ satisfies $1- \tilde{K}_{1,C}>0$ at any point, and the vector fields $X_1$ and $X_2$ defined above, satisfy on $\mathbb{R}^2$
\begin{equation}\label{eq: LV-connec}
\nabla_{X_1}{X_1}=\nabla_{X_1}{X_2}=0,\quad \nabla_{X_2}{X_2}=-\frac{3X_1\tilde{K}_{1,C}}{8\left(1-\tilde{K}_{1,C}\right)}X_1, \quad \nabla_{X_2}{X_1}=\frac{3X_1\tilde{K}_{1,C}}{8\left(1-\tilde{K}_{1,C}\right)}X_2.
\end{equation}
\end{proposition}

\begin{proof}
The Gaussian curvature $\tilde{K}_{1,C}$ of $\left(\mathbb{R}^2,\tilde{g}_{1,C}\right)$ satisfies
\begin{equation*}
\tilde{K}_{1,C}(\rho,\theta)=\tilde{K}_{1,C}(\rho)=\left\{
\begin{array}{ll}
  ^rK(\rho), & (\rho,\theta)\in\left(\rho_{0,r},\rho_{0,r+1}\right)\times\mathbb{R},\qquad r\geq 1 \\\\
  ^0K(\rho), & (\rho,\theta)\in\left(\rho_{0,-1},\rho_{0,1}\right)\times\mathbb{R}\\\\
  ^rK(\rho), & (\rho,\theta)\in\left(\rho_{0,r-1},\rho_{0,r}\right)\times\mathbb{R},\qquad r\leq -1 \\\\
\end{array}
\right.,
\end{equation*}
and
\begin{equation*}
\tilde{K}_{1,C}\left(\rho_{0,r}\right)=\left\{
\begin{array}{ll}
  \lim_{\rho\searrow\rho_{0,r}}\ ^rK(\rho)=\lim_{\rho\nearrow\rho_{0,r}}\ ^{r-1}K(\rho), & r\geq 1 \\\\
  \lim_{\rho\nearrow\rho_{0,r}}\ ^rK(\rho)=\lim_{\rho\searrow\rho_{0,r}}\ ^{r+1}K(\rho), & r\leq -1
\end{array}
\right.,
\end{equation*}
or, more precisely,
\begin{equation*}
\tilde{K}_{1,C}\left(\rho_{0,r}\right)=\left\{
\begin{array}{ll}
  -\frac{1}{9}\xi^{8/3}_{02}+1, & r=2p, \quad p\geq 1 \\\\
  -\frac{1}{9}\xi^{8/3}_{01}+1, & r=2p-1, \quad p\geq 1 \\\\
  -\frac{1}{9}\xi^{8/3}_{01}+1, & r=2p, \quad p\leq -1 \\\\
  -\frac{1}{9}\xi^{8/3}_{02}+1, & r=2p+1, \quad p\leq -1
\end{array}
\right..
\end{equation*}
Thus $1-\tilde{K}_{1,C}>0$ everywhere.

In order to prove \eqref{eq: LV-connec}, we first work on $\left(\mathbb{R}\setminus\left\{\rho_{0,r}\ :\  r\in\mathbb{Z}^\ast\right\}\right)\times\mathbb{R}$ and then, we pass to the limit.
\end{proof}

Now, we can state the following existence and uniqueness result that will play an important role in the next section due to its uniqueness part.

\begin{theorem}\label{th:EUS3}
Let $\left(\mathbb{R}^2,\tilde{g}_{1,C}\right)$ the above complete surface, $C>4/\sqrt{3}$. Then, there exists a unique biconservative immersion $\Phi_{1,C}:\left(\mathbb{R}^2,\tilde{g}_{1,C}\right)\to\mathbb{S}^3$. Moreover, $\grad f_{1,C}\neq 0$ at any point of $\left(\mathbb{R}\setminus\left\{\rho_{0,r}\ :\  r\in\mathbb{Z}^\ast\right\}\right)\times\mathbb{R}$, where $f_{1,C}$ is the mean curvature function of the immersion $\Phi_{1,C}$.
\end{theorem}

\begin{proof}
First, we note that from Proposition \ref{prop-1} we know that the vector fields $X_1$ and $X_2$ on $\mathbb{R}^2$, previously defined, satisfy \eqref{eq: LV-connec} on $\mathbb{R}^2$.

In order to prove the existence of a biconservative immersion $\Phi:\left(\mathbb{R}^2,\tilde{g}_{1,C}\right)\to\mathbb{S}^3$, let us define an operator $A:C\left(T\mathbb{R}^2\right)\to C\left(T\mathbb{R}^2\right)$ by
$$
A\left(X_1\right)=-\frac{\sqrt{1-\tilde{K}_{1,C}}}{\sqrt{3}}X_1, \qquad A\left(X_2\right)=\sqrt{3\left(1-\tilde{K}_{1,C}\right)}X_2.
$$
We will prove that $A$ satisfies the Gauss and the Codazzi equations. Indeed, since the matrix of $A$ with respect to $\left\{X_1,X_2\right\}$ is
\begin{equation*}
A=
\left(
\begin{array}{cc}
          -\frac{\sqrt{1-\tilde{K}_{1,C}}}{\sqrt{3}} & 0 \\
          0 & \sqrt{3\left(1-\tilde{K}_{1,C}\right)}
        \end{array}
\right),
\end{equation*}
it is clear that $\det A=-1+\tilde{K}_{1,C}$, i.e., the Gauss equation is satisfied, and
$$
f_{1,C}=\trace A=\frac{2}{\sqrt{3}}\sqrt{1-\tilde{K}_{1,C}}.
$$
By some direct computations, also using \eqref{eq: LV-connec}, one obtains that
$$
\left(\nabla_{X_1}A\right)\left(X_2\right)=\left(\nabla_{X_2}A\right)\left(X_1\right),
$$
i.e., the Codazzi equation holds on $\mathbb{R}^2$.

Therefore, since $\mathbb{R}^2$ is simply connected, from the Fundamental Theorem of Surfaces in $\mathbb{S}^3$, it follows that there exists a unique, globally defined, isometric immersion $\Phi_{1,C}:\left(\mathbb{R}^2,\tilde{g}_{1,C}\right)\to\mathbb{S}^3$ such that $A$ is its shape operator. Moreover, $\Phi_{1,C}$ is biconservative as the operator $A$ satisfies
$$
A\left(\grad f_{1,C}\right)=-\frac{f_{1,C}}{2}\grad f_{1,C}.
$$
Now, we will prove the uniqueness of the biconservative immersions from $\left(\mathbb{R}^2,\tilde{g}_{1,C}\right)$ in $\mathbb{S}^3$. Let $\Phi_1$ and  $\Phi_2$ two biconservative immersions from $\left(\mathbb{R}^2,\tilde{g}_{1,C}\right)$ in $\mathbb{S}^3$. Obviously, the restrictions of these immersions to $\left(\rho_{0,r},\rho_{0,r+1}\right)\times\mathbb{R}$, with $r\geq 1$, to $\left(\rho_{0,-1},\rho_{0,1}\right)\times\mathbb{R}$, or to $\left(\rho_{0,r-1},\rho_{0,r}\right)\times\mathbb{R}$, with $r\leq -1$, are biconservative. Therefore, using Theorem \ref{thm:reformulate}, it follows that $\grad f_{1}\neq 0$ and $\grad f_{2}\neq 0$ on $\left(\left(\rho_{0,r},\rho_{0,r+1}\right)\times\mathbb{R},\tilde{g}_{1,C}\right)$, with $r\geq 1$, on $\left(\left(\rho_{0,-1},\rho_{0,1}\right)\times\mathbb{R},\tilde{g}_{1,C}\right)$, and on $\left(\left(\rho_{0,r-1},\rho_{0,r}\right)\times\mathbb{R},\tilde{g}_{1,C}\right)$, with $r\leq -1$, and these restrictions are unique (up to isometries of $\mathbb{S}^3$).

It follows that there exists a family of isometries $\left\{^sF\right\}_{s\in\mathbb{Z}}$ of $\mathbb{S}^3$ which preserve its orientation, such that
\begin{equation*}
\begin{array}{ll}
{\Phi_2}_{\left|\left(\rho_{0,s},\rho_{0,s+1}\right)\times\mathbb{R}\right.}= \ ^sF\circ {\Phi_1}_{\left|\left(\rho_{0,s},\rho_{0,s+1}\right)\times\mathbb{R}\right.}, & s\geq 1 \\\\
{\Phi_2}_{\left|\left(\rho_{0,-1},\rho_{0,1}\right)\times\mathbb{R}\right.}= \ ^0F\circ {\Phi_1}_{\left|\left(\rho_{0,-1},\rho_{0,1}\right)\times\mathbb{R}\right.}, & s=0  \\\\
{\Phi_2}_{\left|\left(\rho_{0,s-1},\rho_{0,s}\right)\times\mathbb{R}\right.}= \ ^sF\circ {\Phi_1}_{\left|\left(\rho_{0,s-1},\rho_{0,s}\right)\times\mathbb{R}\right.}, & s\leq -1.
\end{array}
\end{equation*}
Using the same argument as in Theorem 3.10 from \cite{NO19-H}, we can prove that $^sF=\ ^{s+1}F$, for any $s\in\mathbb{Z}$, i.e., we have just one isometry of $\mathbb{S}^3$, so the uniqueness of biconservative immersions from $\left(\mathbb{R}^2,\tilde{g}_{1,C}\right)$ in $\mathbb{S}^3$ is proved.
\end{proof}

\begin{remark}
A more explicit expression of the immersion $\Phi_{1,C}$ was obtained in Theorem 4.18 from \cite{N16}.
\end{remark}

\section{The uniqueness of simply connected, complete non-$CMC$ biconservative surfaces in $N^3(\epsilon)$}

In this section, we will present a complete proof of the uniqueness for the case $\epsilon=-1$. The case $\epsilon=0$ is similar and we will omit its proof. For the last case $\epsilon=1$ we will point out only the differences that appear comparing to the case $\epsilon=-1$.

\subsection{The case $\epsilon=-1$}

We recall that in \cite{NO19-H}, in order to obtain complete biconservative surfaces in $\mathbb{H}^3$, we first constructed a family of abstract complete surfaces $\left(\mathbb{R}^2,\tilde{g}_{-1,C}\right)$, where $C$ is a real constant, and then we obtained

\begin{theorem}[\cite{NO19-H}]\label{th:EUH3}
There exists a unique biconservative immersion
$$
\Phi_{-1,C}:\left(\mathbb{R}^2,\tilde{g}_{-1,C}\right)\to\mathbb{H}^3.
$$
Moreover, $\grad f_{-1,C}\neq 0$ at any point of $\mathbb{R}^\ast\times\mathbb{R}$, where $f_{-1,C}$ is the mean curvature function of the immersion $\Phi_{-1,C}$.
\end{theorem}

\begin{remark}
In the same paper \cite{NO19-H}, an explicit expression of the immersion $\Phi_{-1,C}$ was obtained.
\end{remark}

In order to prove the main result of this section we first give

\begin{proposition} \label{prop-connected-components}
Let $\Phi:M^2\to\mathbb{H}^3$ be a simply connected, complete, non-$CMC$ biconservative surface. Denote $\Omega=\left\{p\in M \ :\  (\grad f)(p)\neq 0\right\}$. We have
\begin{itemize}
  \item [(i)] the open subset $\Omega$ cannot have only one connected component;
  \item [(ii)] the subset $\Omega$ cannot have only two connected components $\Omega_1$ and $\Omega_2$ such that the intersection of their boundaries in $M$, $\partial^M \Omega_1 \cap \partial^M \Omega_2$, is the empty set;
  \item [(iii)] assume that there are two connected components of $\Omega$, $\Omega_1$ and $\Omega_2$ such that $\partial^M \Omega_1 \cap \partial^M \Omega_2=\gamma_0(I)$, where $\gamma_0:I\to M$ is a smooth curve parametrized by arc-length. Then $M$ is isometric to $\left(\mathbb{R}^2,\tilde{g}_{-1,C}\right)$.
\end{itemize}
\end{proposition}

\begin{proof}
In order to prove (i), we assume that $\Omega$ has only one connected component, i.e., $\Omega$ is connected. Then, it follows that $\Omega$ is isometric to an open connected subset $\tilde{\Omega}$ of $\left((0,\infty)\times\mathbb{R},g_{-1,C}\right)$, for some real constant $C$.

If $\partial^M \Omega=\emptyset$, then $\Omega=M$, that means $\tilde{\Omega}$ is isometric to $M$. But this is false since $\tilde{\Omega}$ is not complete.

Assume now that $\partial^M \Omega\neq \emptyset$. First we note that $\grad f \big|_{\partial^M \Omega}=0$. It is known that there exists a unique biconservative immersion from $\left((0,\infty)\times\mathbb{R},g_{-1,C}\right)$, and then from $\left(\mathbb{R}^2,\tilde{g}_{-1,C}\right)$ in $\mathbb{H}^3$, $\Phi_{-1,C}$, which extends the composition between $\Phi \big|_{\Omega}$ and the isometry between $\Omega$ and $\tilde{\Omega}$.

Let $p\in \partial^M\Omega$. It follows that there exists a sequence in $\Omega$ that converges to $p$. This sequence is a Cauchy sequence to whom it corresponds, through the isometry from $\Omega$ to $\tilde{\Omega}$, a sequence in $\tilde{\Omega}$ which is also Cauchy. Now, as $\left(\mathbb{R}^2,\tilde{g}_{-1,C}\right)$ is complete, it follows that the Cauchy sequence in $\tilde{\Omega}$ is also convergent to a point $\tilde{p}\in \overline{\tilde{\Omega}}^{\mathbb{R}^2}\setminus\Int \tilde{\Omega}=\partial ^{\mathbb{R}^2}\tilde{\Omega}$. We note that $\tilde{p}$ is unique determined by $p$ and, moreover, $\left(\grad f\right)(p)=\left(\grad f_{-1,C}\right)\left(\tilde{p}\right)=0$. Analogously, to any point in $\partial ^{\mathbb{R}^2}\tilde{\Omega}$ it corresponds a unique point in $\partial ^{M}\Omega$. Now, since $\tilde{\Omega}$ is an open connected subset of $\left( (0,\infty)\times\mathbb{R},g_{-1,C}\right)$ and $\left(\grad f_{-1,C}\right)\big|_{\partial ^{\mathbb{R}^2}\tilde{\Omega}}=0$, using the completeness of $M$ we obtain that $\tilde{\Omega}=(0,\infty)\times\mathbb{R}$.

We also note that $\partial ^{M}\Omega$ is an injective geodesic curve corresponding to $\theta\to \left(0,\theta\right)$ (see Remark 3.8 in \cite{NO19-H}). The boundary $\partial ^{M}\Omega$ is a closed, non-compact subset of $M$ and, moreover, it is a regular curve of $M$.

Further, let us consider $p_0\in \partial^M\Omega$ and $\gamma(t)$ a geodesic curve parametrized by arc-length which is directed towards the exterior of $\Omega$ such that $\gamma(0)=p_0$ and $\gamma'(0)$ is orthogonal to $\partial^M\Omega$. We take a sequence $p_n=\gamma\left(1/n\right)$, $n\in\mathbb{N}^\ast$, that converges to $p_0$. It is clear that $\left(p_n\right)\subset \Int (M\setminus\Omega)=(M\setminus\Omega)\setminus\left(\partial^M \Omega\right)$ and therefore $(\Delta f)\left(p_n\right)=(\grad f)\left(p_n\right)=0$. Passing to the limit for $n\to\infty$, one obtains $(\Delta f)\left(p_0\right)=(\grad f)\left(p_0\right)=0$.

On the other hand, from \eqref{eq:bicons1} with $\epsilon=-1$, which holds on $\Omega$, passing to the limit we get
$$
f\left(p_0\right)(\Delta f) \left(p_0\right)+|(\grad f)\left(p_0\right)|^2-\frac{4}{3}f^2\left(p_0\right)-f^4\left(p_0\right)=0,
$$
and, therefore $f\left(p_0\right)=0$. But this is a contradiction because
$$
f\left(p_0\right)=f_{-1,C}\left(0,\theta_0\right)=\frac{2}{3\sqrt{3}}\xi_{01}^{4/3}\neq 0.
$$
Now, to prove (ii), we assume that $\Omega$ has exactly two connected components $\Omega_1$ and $\Omega_2$ such that $\partial^M \Omega_1 \cap \partial^M \Omega_2=\emptyset$. As $\Omega_1$ is maximal it follows that $\grad f\big|_{\partial^{M}\Omega_1}=0$. Using the same ideas as in the proof of (i), we get that $\Omega_1$ is isometric to $\tilde{\Omega}_1=(0,\infty)\times\mathbb{R}$ and $\partial ^{M}\Omega_1$ is an injective geodesic curve corresponding to $\theta\to \left(0,\theta\right)$.

Further, we consider $\gamma(t)$ a geodesic curve parametrized by arc-length which is directed towards the exterior of $\Omega_1$ and starts on the boundary of $\Omega_1$. Since $\partial^M \Omega_1 \cap \partial^M \Omega_2=\emptyset$, we can choose a sequence $p_n=\gamma\left(1/n\right)\in \Int (M\setminus\Omega)=(M\setminus\Omega)\setminus\left(\partial^M \Omega\right)$, for any $n\in\mathbb{N}^\ast$,  and, as in (i), we obtain a contradiction.

In the last part of the proof, we will show that if we assume that there are two connected components of $\Omega$, $\Omega_1$ and $\Omega_2$, such that $\partial^M \Omega_1 \cap \partial^M \Omega_2=\gamma_0(I)$, where $\gamma_0:I\to M$ is a curve parametrized by arc-length, then $M$ is isometric to $\left(\mathbb{R}^2,\tilde{g}_{-1,C}\right)$. First, we note that, since $\partial^M \Omega_1$ and $\partial^M \Omega_2$ are geodesics and they coincide along a curve, they coincide everywhere, i.e., $\partial^M \Omega_1=\partial^M \Omega_2$. As $\Omega_1$ is isometric to $\left((0,\infty)\times\mathbb{R},g_{-1,C}\right)$ and $\Omega_2$ is isometric to $\left((0,\infty)\times\mathbb{R},g_{-1,C'}\right)$, from the continuity of the Gaussian curvature on $\partial^M \Omega_1$ we get $C=C'$.

It follows that $M=\Omega_1\cup \partial^M\Omega_1\cup\Omega_2$ and it is isometric to $\left(\mathbb{R}^2,\tilde{g}_{-1,C}\right)$.



\end{proof}

Our result concerning the uniqueness is

\begin{theorem}\label{th:uniqueness1}
Let $\Phi:M^2\to\mathbb{H}^3$ be a simply connected, complete, non-$CMC$ biconservative surface. Then, up to isometries of the domain and codomain, $M$ and $\Phi$ are those given in Theorem \ref{th:EUH3}.
\end{theorem}

\begin{proof}
Let $\Phi:M^2\to\mathbb{H}^3$ be a simply connected, complete, non-$CMC$ biconservative surface. Then, from Proposition \ref{prop-connected-components} we have that $\Omega$ has two connected components, isometric one to another, with a common boundary and $M$ is isometric to $\left(\mathbb{R}^2,\tilde{g}_{-1,C}\right)$. Then, from Theorem \ref{th:EUH3} we conclude.
\end{proof}

\subsection{The case $\epsilon=0$}

We have the following uniqueness result

\begin{theorem}\label{th:uniqueness2}
Let $\Phi:M^2\to\mathbb{R}^3$ be a simply connected, complete, non-$CMC$ biconservative surface. Then, up to isometries of the domain and codomain, $M=\left(\mathbb{R}^2,g_{0,C}=C(\cosh u)^6\left(du^2+dv^2\right)\right)$ and
$$
\Phi(u,v)=\left(\frac{\sqrt{C}}{3}\left(\cosh u\right)^3 \cos (3v), \frac{\sqrt{C}}{3}\left(\cosh u\right)^3 \sin (3v), \frac{\sqrt{C}}{2}\left(\frac{1}{2}\sinh (2u)+u\right)\right),
$$
where $C$ is a positive real constant.
\end{theorem}

\begin{remark}
The expressions of $M$ and $\Phi$ were given in \cite{N16} and the proof of the above theorem is similar to the proof of Theorem \ref{th:uniqueness1}, $\Omega$ having two connected components.
\end{remark}

\subsection{The case $\epsilon=1$}

We have the following uniqueness result

\begin{theorem}\label{th:uniqueness3}
Let $\Phi:M^2\to\mathbb{S}^3$ be a simply connected, complete, non-$CMC$ biconservative surface. Then, up to isometries of the domain and codomain, $M$ and $\Phi$ are those given in Theorem \ref{th:EUS3}.
\end{theorem}

\begin{proof}
The proof follows the same lines as the proof of Theorem \ref{th:uniqueness1}. We have a similar result as in Proposition \ref{prop-connected-components} (i), since here $f^2_{1,C}\left(\rho_{0,-1},\theta_0\right)$ and $f^2_{1,C}\left(\rho_{0,1},\theta_0\right)$  cannot be $0$ or $4/3$. But here, the set $\Omega$ has a countable number of connected components $\left\{\Omega_s\right\}_{s\in\mathbb{Z}}$, any two of them being isometric one to another. The boundary of any connected component $\Omega_s$ is made up by two distinct injective geodesic curves and the boundaries of $\Omega_s$ and $\Omega_{s+1}$ have in common one injective geodesic curve. Therefore, one obtains that $M$ is isometric to $\left(\mathbb{R}^2,\tilde{g}_{1,C}\right)$ and from Theorem \ref{th:EUS3} we conclude.
\end{proof}

\begin{remark}
As an immediate consequence of Theorems \ref{th:uniqueness1}, \ref{th:uniqueness2} and \ref{th:uniqueness3}, it follows that a simply connected, complete, non-$CMC$ surface in $N^3(\epsilon)$, with the topologic interior of the set $\{p\in M\ :\  (\grad f)(p)=0\}$ being non-empty, cannot be biconservative.
\end{remark}

As we have already mentioned in \cite{NO19-H}, from Theorems \ref{th:uniqueness1} and \ref{th:uniqueness2} we deduce that

\begin{theorem}
If $M$ is a compact biconservative surface in $N^3(\epsilon)$, $\epsilon \in\{-1,0\}$, then $M$ is $CMC$.
\end{theorem}

\end{document}